\documentclass[11pt]{amsart}

\usepackage{amsmath}
\usepackage{amsfonts}   
\usepackage{amsthm}     

\usepackage{enumitem}
\setlist[enumerate,1]{label=\alph*),ref=\theenumi}

\newtheorem{lemma}{Lemma}[section]
\newtheorem{proposition}[lemma]{Proposition}
\newtheorem{theorem}[lemma]{Theorem}

\usepackage{todonotes}
\newcommand{\R}{\mathbb{R}}                             
\newcommand{\toweak}{\rightharpoonup}                   

\newcommand{\Lcal}{\mathcal{L}}                         
\newcommand{\Jcal}{\mathcal{J}}                         

\newcommand{\embed}{\hookrightarrow}

\newcommand{\DefaultGfunction}{G}                       

\DeclareMathOperator{\LspaceSymbol}{\mathbf{L}} 
\newcommand{\Lpspace}[1][p]{\LspaceSymbol^{{#1}}}
\newcommand{\LGspace}[1][\DefaultGfunction]{\LspaceSymbol^{{#1}}}
\newcommand{\LGastspace}[1][\DefaultGfunction]{\LspaceSymbol^{{#1^\star}}}

\DeclareMathOperator{\WspaceSymbol}{\mathbf{W}} 
\newcommand{\Wspace}[1]{\WspaceSymbol^{#1}}
\newcommand{\WLGspace}[1][\DefaultGfunction] { {\WspaceSymbol^1}\LspaceSymbol^{{#1}} }
\newcommand{\WLGzspace}[1][\DefaultGfunction]{ {\WspaceSymbol^1_0}\LspaceSymbol^{{#1}}}

\newcommand{\norm}[1]{\|#1\|}                           
\newcommand{\inner}[2]{\left\langle #1,#2\right\rangle} 

\newcommand{\Lpnorm}[2][p]{\norm{#2}_{\Lpspace[{#1}]}}
\newcommand{\LGnorm}[2][\DefaultGfunction]{\norm{#2}_{\LGspace[{#1}]}}
\newcommand{\LGastnorm}[2][{\DefaultGfunction}]{\norm{#2}_{\LGastspace[{#1}]}}
\newcommand{\WLGnorm}[2][\DefaultGfunction]{\norm{#2}_{\WLGspace[{#1}]}}
\newcommand{\WLGznorm}[2][\DefaultGfunction]{\norm{#2}_{\WLGzspace[{#1}]}}

\begin{document}


\title[Mountain pass solutions with general anisotropic operator]{Mountain pass solutions to Euler-Lagrange equations with general anisotropic operator}

\begin{abstract}
Using the Mountain Pass Theorem we show that the problem
\begin{equation*}
\begin{cases}
\frac{d}{dt}\Lcal_v(t,u(t),\dot u(t))=\Lcal_x(t,u(t),\dot u(t))\quad \text{ for a.e. }t\in[a,b]\\
u(a)=u(b)=0
\end{cases}
\end{equation*}
has a solution in anisotropic Orlicz-Sobolev space. We consider Lagrangian $\Lcal=F(t,x,v)+V(t,x)+\langle f(t), x\rangle$ with growth condition determined by anisotropic G-function and some geometric condition of Ambrosetti-Rabinowitz type.
\end{abstract}

\author{M. Chmara}

\author{J. Maksymiuk}

\address{
Department of Technical Physics and Applied Mathematics,
Gda\'{n}sk University of Technology,
Narutowicza 11/12, 80-233 Gda\'{n}sk, Poland
}
\email{magdalena.chmara@pg.edu.pl, jakub.maksymiuk@pg.edu.pl}


\maketitle


\section{Introduction}

We consider the  second order boundary value problem:
\begin{equation}
\tag{ELT}
\label{eq:ELT}
\begin{cases}
\frac{d}{dt}\Lcal_v(t,u(t),\dot u(t))=\Lcal_x(t,u(t),\dot u(t))\quad \text{ for a.e. }t\in [a,b]\\
u(a)=u(b)=0
\end{cases}
\end{equation}
where $\Lcal\colon [a,b]\times\R^N\times\R^N\to \R$ is given by
\begin{equation*}
\Lcal(t,x,v)=F(t,x,v)+V(t,x)+\inner{f(t)}{x}.
\end{equation*}
Using the Mountain Pass Theorem we show that the problem \eqref{eq:ELT} has a solution in anisotropic Orlicz-Sobolev space.

Recently, existence periodic solution to the equation
\[
\frac{d}{dt}\nabla G(\dot u(t))=\nabla V(t,u)+f(t)
\]
was established by Authors in \cite{ChmMak19} via Mountain Pass Theorem. In this paper we consider more general differential operator
\[
\frac{d}{dt} F_v(t,u,\dot{u})
\]
We assume that $F$ is convex in the last variable and that the growth of $F$ and its derivatives is determined by underlying G-function. We also assume that $F$ and $V$ satisfies some geometric conditions of Ambrosetti-Rabinowitz type.

If $F(v)=\frac1p|v|^p$ then the equation \eqref{eq:ELT} reduces to $p$-laplacian equation
$\frac{d}{dt}(|\dot{u}|^{p-2}\,\dot{u})=\nabla V(t,u)+f(t)$. One can also consider more general case $F(v)=\phi(|v|)$, where $\phi$ is convex and nonnegative. In the above cases $F$ does not depend on $v$ directly but rather on its norm $|v|$ and the growth of $F$ is the same in all directions, i.e. $F$ has isotropic growth. Equation \eqref{eq:ELT} with Lagrangian $L(t,x,v)=\frac{1}{p}|v|^p+V(t,x)+\inner{f(t)}{x}$ has been studied by many authors under different conditions, for example in \cite{MawWil89,IzyJan05,Dao16}.

The novelty of this article lies in fact that $F$ can be depended not only on $\dot{u}$ but also on $t$ and $u$. More over we consider anisotropic case, i.e. $F(t,x,\cdot)$ depends on all components of $v$ not only on $|v|$ and has different growth in different directions, which generalizes previous results, for example \cite{ChmMak19}, where kinetic part was given by an anisotropic G-function.

We obtain solution to the problem \eqref{eq:ELT} by applying the Mountain Pass Theorem. To do this we first need to show that corresponding action functional satisfies the Palais-Smale condition. First we prove that a Palais-Smale sequence $\{u_n\}$ is bounded, the proof is rather standard and involves Ambrosetti-Rabinowitz condition. Then we need to show that $\{u_n\}$ has convergent sequence. We show that
\[
\lim_{n\to \infty} \int_I \inner{F_v(t,u_n,\dot{u}_n}{\dot{u}-\dot{u}_n}\,dt = 0,
\]
where $u$ is a weak limit of $\{u_n\}$, which in turn implies that
\[
\lim_{n\to \infty} \int_I F(t,u_n,\dot{u}_n)\,dt = \int_I F(t,u,\dot{u})\,dt.
\]
The proof of this fact is based on convexity of $F$ and embedding $\WLGspace \embed \Lpspace[\infty]$.
Next, using convexity of $F$ and condition $F(t,x,v)\geq \Lambda G(v)$, that $\{\dot{u}_n\}$ converges strongly. This reasoning shows that action functional satisfies so called $(S_+)$ condition (see for example \cite{DraMil07}).

This result seems to be of independent interest and the methods presented in this paper can be also applied in other problems (e.g. in the case of periodic problem).

Our work was partially inspired by the work of de Napoli and Mariani \cite{DeNapMar03}. They consider elliptic PDE
\[
-div (a(x,\nabla u))=f(x,u)
\]
with Dirichlet conditions. To show that corresponding functional satisfies the Palais-Smale condition they also prove that $(S_+)$ condition is satisfied. However, they use stronger condition, namely they assume uniform convexity of functional.

As in \cite{ChmMak19} we consider two cases: $G$ satisfying $\Delta_2$, $\nabla_2$ at infinity and globally. It turns out, that in both cases the mountain pass geometry of action functional is strongly depended on two factors: the embedding constant for $\WLGspace \embed \Lpspace[\infty]$ and on Simonenko indices $p_G$ and $q_G$ (see Lemmas \ref{lem:J(e)<0} and \ref{lem:J>0}).

Similar observation can be found in \cite{Cle04, Bar17, Cle00} where the existence of elliptic systems via the  Mountain Pass Theorem is considered. In \cite{Bar17} authors deal with an anisotropic problem. The isotropic case is considered in \cite{Cle04, Cle00}.

\section{ Orlicz-Sobolev spaces }

In this section we briefly recall the notion of anisotropic Orlicz-Sobolev spaces. For more details we refer the reader to \cite{ChmMak17, ChmMak19} and references therein. We assume that
\begin{enumerate}[ref=(G),label=(G)]
\item \label{asm:G}
$G\colon \R^N\to [0,\infty)$ is a continuously differentiable G-function (i.e. $G$ is convex, even, $G(0)=0$ and $G(x)/|x|\to \infty$ as $|x|\to \infty$ ) satisfying $\Delta_2$ and $\nabla_2$ conditions (at infinity).
\end{enumerate}

Typical examples of such $G$ are: $G(x)=|x|^p$, $G(x_1,x_2)=|x_1|^{p_1} + |x_2|^{p_2}$ and $G(x)=|x|^p \log(1+|x|)$, $1<p_i<\infty$, $1<p<\infty$.

Let $I=[a,b]$. The Orlicz space associated with $G$ is defined to be
\begin{equation*}
\LGspace=\LGspace(I,\R^N) = \left\{u\colon I\to \R^N\colon \int_I G(u)\,dt < \infty  \right\}.
\end{equation*}
The space $\LGspace$ equipped with the Luxemburg norm
\begin{equation*}
\LGnorm{u} = \inf\left\{\lambda>0\colon \int_I G\left(\frac{u}{\lambda}\right)\,dt\leq 1  \right\}
\end{equation*}
is a separable, reflexive Banach space. We have two important inequalities:
\begin{enumerate}
\item the Fenchel inequality
\begin{equation*}
\inner{u}{v} \leq G(u)+G^\ast(v), \text{ for every $u$, $v\in \R^N$,}
\end{equation*}

\item the H\"older inequality
\begin{equation*}
\int_I \inner{u}{v}\,dt \leq 2\LGnorm{u} \LGastnorm{v}, \text{ for every $u\in \LGspace$ and $v\in \LGastspace$,}
\end{equation*}
\end{enumerate}
where $\DefaultGfunction^\ast$ is a convex conjugate of $G$. Functional $R_G(u)=\int_I G(u)\,dt$ is called modular. Note that if $G(x)=|x|^p$ then $\LGspace=\Lpspace$ and $R_G(u)=\Lpnorm{u}^p$. In general case, relation between modular and the Luxemburg norm is more complicated.

Define the Simonenko indices for G-function
\[
p_G=\inf_{|x|>0}\frac{\inner{x}{\nabla G(x)}}{G(x)},
\quad
q_G=\sup_{|x|>0}\frac{\inner{x}{\nabla G(x)}}{G(x)},
\]
It is obvious that $p_G\leq  q_G$. Moreover, since $G$ satisfies $\Delta_2$ and $\nabla_2$, $1<p_G$ and $q_G<\infty$. If $G(x)=\tfrac{1}{p}|x|^p$ then $p_G=q_G=p$. The following results are crucial to Lemma \ref{lem:J>0}
\begin{proposition}
\label{prop:modular_greater_than_norm_globally}
Assume that $G$ satisfies $\Delta_2$ and $\nabla_2$ globally.
\begin{enumerate}
\item If $\LGnorm{u}\leq 1,$ then $\LGnorm{u}^{q_G}\leq R_G(u).$
\item If $\LGnorm{u}>1,$ then $\LGnorm{u}^{p_G}\leq R_G(u).$
\end{enumerate}
\end{proposition}
The proof can be found in \cite[Appendix A]{ChmMak19}. More information about indices for isotropic case can be found in \cite{Mal89,Cle04}. In case $G$ satisfies $\Delta_2$ and $\nabla_2$ only at infinity we have weaker estimates

\begin{proposition}
\label{prop:modular_greater_than_norm}
If $\LGnorm{u}>1$ then $R_G(u) \geq \LGnorm{u}$. If $\LGnorm{u}\leq 1$ then $R_G(u) \leq \LGnorm{u}$.
\end{proposition}

For relations between Luxemburg norm and modular for anisotropic spaces we refer the reader to \cite[Examples 3.8 and 3.9]{ChmMak17}. We will also use the following simple observations

\begin{lemma}
\label{lem:coercivityofmodular}
\begin{equation*}
\lim_{\LGnorm{u}\to \infty}\frac{R_G(u)}{\LGnorm{u}} = \infty.
\end{equation*}
\end{lemma}

\begin{lemma}
\label{lem:norm_modular_bounded}
Let $\{u_n\}\subset \LGspace$. Then $\{u_n\}$ is bounded if and only if $\{R_G(u_n)\}$ is bounded.
\end{lemma}

The anisotropic Orlicz-Sobolev space is defined to be
\begin{equation*}
\WLGspace = \WLGspace(I,\R^N) = \{u\in \LGspace\colon \dot{u}\in \LGspace\},
\end{equation*}
with usual norm
\begin{equation*}
\WLGnorm{u} = \LGnorm{u} + \LGnorm{\dot{u}}.
\end{equation*}
It is known that elements of $\WLGspace$ are absolutely continuous functions.
An important role in our considerations plays an embedding  constant for $\WLGspace\embed \Lpspace[\infty]$. We denote this constant by $C_{\infty,G}$. Let $A_G\colon \R^N\to [0,\infty)$ be the greatest convex minorant of $G$ (see \cite{AciMaz17}), then
\[
\Lpnorm[\infty]{u} \leq \max\{1,|I|\} A_G^{-1}\left(\frac{1}{|I|}\right)\WLGnorm{u}.
\]
We introduce the following subspace of $\WLGspace$:
\begin{equation*}
\WLGzspace = \{u\in \WLGspace\colon u=0 \text{ on } \partial I \}.
\end{equation*}

It is proved in \cite[Theorem 4.5]{ChmMak17} that for every $u\in \WLGzspace$ the following form of Poincar\'
e inequality holds
\begin{equation}
\label{ineq:Poincare}
\LGnorm{u}\leq |I|\, \LGnorm{\dot{u}}.
\end{equation}
It follows that one can introduce an equivalent norm on $\WLGzspace$:
\begin{equation*}
\WLGznorm{u} = \LGnorm{\dot{u}}.
\end{equation*}

\section{Main results}

Let $G$ satisfies assumption \ref{asm:G} and let $I=[a,b]$. We consider Lagrangian $\Lcal\colon I\times\R^N\times\R^N\to \R$ given by
\[
\Lcal(t,x,v)=F(t,x,v)+V(t,x)+\inner{f(t)}{x}.
\]
We assume that $F\colon I\times\R^N\times\R^N\to \R$, $V\colon I\times\R^N\to \R$ are of class $C^1$ and satisfy
\begin{enumerate}[label=(F$_\arabic*$),ref=(F$_\arabic*$)]
\item
\label{asm:F:convex}
$F(t,x,\cdot)$ is convex for all $(t,x)\in I\times \R^N$,
\item
\label{asm:F:growth}
there exist $a\in C(\R_+,\R_+)$ and $b\in\Lpspace[1](I,\R_+)$ such that for all $(t,x,v)\in I\times \R^N\times\R^N$:
\begin{gather}
\label{F:growth:F}
|F(t,x,v)|\leq a(|x|)\,(b(t)+G(v)),\\
\label{F:growth:Fx}
|F_x(t,x,v)\leq a(|x|)\,(b(t)+G(v)), \\
\label{F:growth:Fv}
G^{\ast}(F_v(t,x,v))\leq a(|x|)\,(b(t)+\,G^{\ast}\left(\nabla G(v)\right)),
\end{gather}
\item \label{asm:F:AR}
There exist $\theta_{F} > 0 $ such that for all $(t,x,v)\in I\times \R^N\times\R^N$:
\begin{gather*}
\inner{F_x(t,x,v)}{x} + \inner{F_v(t,x,v)}{v} \leq \theta_{F}\, F(t,x,v),
\end{gather*}
\item \label{asm:F:ellipticity}
there exists $\Lambda>0$ such that for all $(t,x,v)\in I\times \R^N\times\R^N$:
\begin{equation*}
F(t,x,v) \geq \Lambda\, G(v),
\end{equation*}
\item \label{asm:F:zero}
$F(t,x,0) = 0$ for all $(t,x)\in I\times \R^N$,
\end{enumerate}


\begin{enumerate}[label=(V$_\arabic*$),ref=(V$_\arabic*$)]
\item \label{asm:V:AR}
there exist $\theta_V>1$, $\theta_V > \theta_F $ and $r_0>0$ such that for all $t\in I$
\begin{equation*}
\inner{\nabla V(t,x)}{x} \leq \theta_V V(t,x),\ |x|\geq r_0,
\end{equation*}

\item \label{asm:V:zero}
\begin{equation*}
\int_I V(t,0)\,dt=0,
\end{equation*}

\item
\label{asm:V:near_zero}
there exists $\rho_0>0$ and $g\in \Lpspace[1](I,\R)$ such that for all $t\in I$
\begin{equation*}
V(t,x)\geq -g(t),\ |x|\leq \rho_0,
\end{equation*}

\item \label{asm:V:negative}
\begin{equation*}
V(t,x) < 0,\ t\in I,\ |x|\geq r_0,
\end{equation*}
\end{enumerate}

\begin{enumerate}[label=(f)]
\item \label{asm:f}
$f\in \LGastspace(I,\R^N)$.
\end{enumerate}

Now we can state our main theorems.

\begin{theorem}
\label{thm:A}
Assume that $\rho_0\geq C_{\infty,G}$ and
\begin{equation}
\tag{A}
\label{thmA:asm_ineq}
\int_I g(t)\,dt < (\Lambda -2|I|\,\LGastnorm{f}) \frac{\rho_0}{C_{\infty,G}}.
\end{equation}
Then \eqref{eq:ELT} has at least one nontrivial solution.
\end{theorem}

Assumption $\rho_0\geq C_{\infty,G}$ can be relaxed if we assume that $G$ satisfies $\Delta_2$ and $\nabla_2$ globally. In this case we also have weaker assumptions on $V$.

\begin{theorem}
\label{thm:B}
Assume that $G$ satisfies $\Delta_2$ and $\nabla_2$ globally
and
\begin{equation}
\tag{B}
\label{thmB:asm_ineq}
\int_I g(t)\,dt + 2|I|\,\LGastnorm{f} \frac{\rho_0}{C_{\infty,G}} < \Lambda
\begin{cases}
\left(\frac{\rho_0}{C_{\infty,G}}\right)^{q_G},& \rho_0\leq C_{\infty,G}\\
\left(\frac{\rho_0}{C_{\infty,G}}\right)^{p_G},& \rho_0>C_{\infty,G}
\end{cases}
\end{equation}
Then \eqref{eq:ELT} has at least one nontrivial solution.
\end{theorem}

One can show that, in fact, every solution of \eqref{eq:ELT} is of class $\Wspace{1,\infty}$ (see \cite[Proposition 3.5]{ChmMak19}).

\subsection{Some remarks on assumptions}

Assumptions \ref{asm:F:AR} and \ref{asm:V:AR} are Ambrosetti-Rabinowitz type conditions. It follows that $F$ and $V$ are subhomogeneous respectively everywhere and
for large arguments (cf. \cite{DeNapMar03}).
\begin{lemma}
\label{lem:subhomogeneous}
For every $\lambda>1$
\begin{enumerate}
\item \begin{equation*}
F(t,\lambda x, \lambda v)\leq \lambda^{\theta_F} F(t,x,v)\
\text{ for all $(t,x,v)\in I\times \R^N\times \R^N$}
\end{equation*}

\item
\begin{equation*}
V(t,\lambda x)\leq \lambda^{\theta_V} V(t,x)\ \text{ for all $t\in I$, $|x|\geq r_0$}
\end{equation*}
\end{enumerate}
\end{lemma}
\begin{proof}
Let $(t,x,v)\in I\times \R^N\times \R^N$ and $\lambda>1$, then
\begin{multline*}
\log\left(\frac{F(t,\lambda x,\lambda v)}{F(t,x,v)}\right)
=
\int_1^\lambda \frac{d}{d\lambda} \log F(t,\lambda x,\lambda v)\,d\lambda
= \\ =
\int_1^\lambda \frac{\inner{F_x(t,\lambda x,\lambda v)}{x} + \inner{F_v(t,\lambda x, \lambda v)}{v}}{F(t,\lambda x,\lambda v)} \,d\lambda
\leq
\int_1^\lambda  \frac{\theta_{F}}{\lambda} d\lambda
=
\log \lambda^{\theta_F}.
\end{multline*}
by \ref{asm:F:AR} and the result follows. The proof of b) is similar
\end{proof}

\section{Proof of the main theorems}

Define action functional $\Jcal\colon \WLGzspace(I,\R^N)\to \R$ by
\begin{equation}
\label{eq:J}\tag{$\Jcal$}
\Jcal(u)=\int_I F(t,u,\dot{u}) + V(t,u) + \inner{f}{u}\,dt
\end{equation}
Under above assumptions, $\Jcal$ is well defined and of class $C^1$. Furthermore, its derivative is given by
\begin{equation}
\label{eq:J'}\tag{$\Jcal'$}
\Jcal'(u)\varphi =
\int_I \inner{F_x(t,u,\dot{u})}{\varphi}\,dt +
\int_I \inner{F_v(t,u,\dot{u})}{\dot{\varphi}}\,dt +
\int_I \inner{\nabla V (t,u)}{\varphi} + \inner{f}{\varphi}\, dt
\end{equation}
See \cite[Theorem 5.7]{ChmMak17} for more details. It is standard to prove that critical points of $\Jcal|_{\WLGzspace}$ are solutions of \eqref{eq:ELT}.

Our proof is based on the well-known Mountain Pass Theorem (see \cite{AmbRab73}).

\begin{theorem}
\label{thm:mpt}
Let $X$ be a real Banach space and $I\in C^1(X,\R)$ satisfies the following conditions:
\begin{enumerate}
    \item $I$ satisfies Palais-Smale condition,
    \item $I(0)=0$,
    \item there exists $\alpha>0$ such that $I\rvert_{\partial B_{\rho}(0)}\geq \alpha$,
   \item there exist $\rho>0$, $e\in X$ such that $\norm{e}_X>\rho$ and $I(e)<0$.
\end{enumerate}
Then $I$ possesses a critical value  $c\geq\alpha$ given by
$
c=\inf_{g\in\Gamma}\max_{s\in[0,1]}I(g(s)),
$
where $\Gamma=\{g\in C([0,1],X):~~ g(0)=0,~ g(1)=e\}.$
\end{theorem}

We divide the proof into sequence of lemmas.

\subsection{The Palais-Smale condition}

Now we show that $\Jcal$ satisfies the Palais-Smale condition. We divide the proof into two steps. First we show that every (PS)-sequence is bounded and then that it contains a convergent subsequence.

The first part of the proof is standard. Let us note that assumptions \ref{asm:F:AR}, \ref{asm:F:ellipticity} and \ref{asm:V:AR} are crucial.
The second part is more involved, let us outline it. First we show that $u_n\toweak u$ and embedding $\WLGspace\embed \Lpspace[\infty]$ imply that
\[
\int_I \inner{F_v(t,u_n,\dot{u}_n)}{\dot{u}-\dot{u}_n}\,dt \to 0.
\]
Then we show that
\[
\int_I F(t,u_n,\dot{u}_n)\,dt\to \int_I F(t,u,\dot{u})\,dt.
\]
In the case of p-Laplacian equation (i.e. $F(t,x,v)=\frac{1}{p}|v|^p$), the last condition implies that $\dot{u}_n\to \dot{u}$. The same is true if $F(t,u,v)=G(v)$, since in this case $R_G(\dot{u}_n)\to R_G(\dot{u})$. The last condition implies desired convergence for $\{\dot{u}_n\}$ (see \cite[Lemma 3.16]{ChmMak17} and \cite[p. 593]{ChmMak19}).

In our case this argument does not apply directly because convergence of above integrals does not imply that $R_G(\dot{u}_n)\to R_G(\dot{u})$. However, we can extend the reasoning presented in the proof of \cite[Lemma 3.16]{ChmMak17} to our general integrand and show that
\[
\int_I F\left(t,u_n,\frac{\dot{u}_n-\dot{u}}{2}\right)\,dt  \to 0
\]
and then apply condition \ref{asm:F:ellipticity} to show that $R_G(\dot{u_n}-\dot{u})\to 0$ and hence $\dot{u}_n\to \dot{u}$ in $\LGspace$ by \cite[Lemma 3.13]{ChmMak17}.

\begin{lemma}
    \label{lem:ps}
Functional $\Jcal$ satisfies the Palais-Smale condition
\end{lemma}
\begin{proof}
Fix $u\in \WLGzspace$.
 From assumptions \ref{asm:F:AR} and \ref{asm:F:ellipticity} we obtain
\begin{multline}
\label{lem:PS:F_below}
\int_I \theta_V F(t,u,\dot{u}) - \inner{F_x(t,u,\dot{u})}{u} - \inner{F_v(t,u,\dot{u})}{\dot{u}}\,dt
\geq \\ \geq
(\theta_V-\theta_{F}) \int_I F(t,u,\dot{u}) \,dt
\geq
C_1 \int_I G(\dot{u}) \,dt,
\end{multline}
where $C_1=\Lambda(\theta_V-\theta_{F})>0$, since $\theta_V>\theta_F$.
Set $M=\sup\{ |\theta_V V(t,x)-\inner{\nabla V(t,x)}{x}| \colon t\in I, |x|\leq r_0 \}$, then by \ref{asm:V:AR} we obtain
\begin{equation}
\label{lem:PS:V_below}
\int_I \theta_V V(t,u) - \inner{\nabla V(t,u)}{u} \,dt
\geq
\int_{\{ |u(t)|>r_0 \}} \theta_V V(t,u) - \inner{\nabla V(t,u)}{u} \,dt - |I| M
\geq
- |I| M.
\end{equation}
We also have, by H\"older's inequality, \eqref{ineq:Poincare} and \ref{asm:f}, that
\begin{equation}
\label{lem:PS:f_below}
(\theta_V-1) \int_I \inner{f(t)}{u} \,dt
\geq
-2(\theta_V-1)\LGastnorm{f} \LGnorm{u}
\geq
-  C_2 \LGnorm{\dot{u}},
\end{equation}
where $C_2=2 |I|(\theta_V-1)\LGastnorm{f}>0$.
From \eqref{eq:J} and \eqref{eq:J'} we get
\begin{multline*}
\theta_V \Jcal(u) - \Jcal'(u)u
=
\int_I \theta_V F(t,u,\dot{u}) - \inner{F_x(t,u,\dot{u})}{u} - \inner{F_v(t,u,\dot{u})}{\dot{u}}\,dt
 + \\ +
\int_I \theta_V V(t,u) - \inner{\nabla V(t,u)}{u} \,dt
+
(\theta_V-1) \int_I \inner{f(t)}{u} \,dt.
\end{multline*}

Using \eqref{lem:PS:F_below}, \eqref{lem:PS:V_below} and \ref{lem:PS:f_below} we obtain
\begin{equation*}
C_1 \int_I G(\dot{u}) \,dt
\leq
\theta_V |\Jcal(u)| + \norm{\Jcal'(u)} \WLGznorm{u} + C_2 \LGnorm{\dot{u}} + |I| M.
\end{equation*}
Let $\{u_n\}\subset \WLGzspace$ be a Palais-Smale sequence, i.e. $\{\Jcal(u_n)\}$ is bounded and $\Jcal'(u_n)\to 0$.  If $\{u_n\}$ is not bounded, we may assume that $\WLGznorm{u_n}\to \infty$. Then dividing by $\WLGznorm{u_n}=\LGnorm{\dot{u}_n}$ we get
\begin{equation*}
\frac{C_1}{\LGnorm{\dot{u}_n}}\int_I G(\dot{u}_n)\, dt
\leq
\frac{\theta_V |\Jcal(u_n)|}{\LGnorm{\dot{u}_n}}
+
\norm{\Jcal'(u_n)} + C_2 + \frac{|I| M}{\LGnorm{\dot{u}_n}}.
\end{equation*}
Letting $n\to \infty$, we obtain a contradiction with Lemma \ref{lem:coercivityofmodular}, thus $\{u_n\}$ is bounded.

Next we show that $\{u_n\}$ has a convergent subsequence.  Passing to a subsequence if necessary, we
%
may assume that $u_n\to u$ in $\Lpspace[\infty]$, $\{\dot{u}_n\}$ bounded in $\LGspace$,  $\dot{u}_n\to \dot{u}$ a.e. and $u_n\to u$ a.e.

Since $\Jcal'(u_n)\to 0$ and $\{u_n-u\}$ is bounded in $\WLGzspace$, we conclude that
\begin{equation*}
\lim_{n\to\infty} \inner{\Jcal'(u_n)}{u_n-u} = 0,
\end{equation*}
from the other hand, 
\begin{equation*}
\lim_{n\to\infty}\int_I \inner{\nabla V(t,u_n)+f(t)}{u_n-u} \,dt = 0.
\end{equation*}
Thus, by \eqref{eq:J'} we have that
\begin{equation*}
\lim_{n\to\infty}
\int_I \inner{F_x(t,u_n,\dot{u}_n)}{u_n-u}\,dt +
\int_I \inner{F_v(t,u_n,\dot{u}_n)}{\dot{u}_n-\dot{u}}\,dt
= 0
\end{equation*}

Define nondecreasing function $\alpha(s)=\sup_{\tau\in[0,s]}a(\tau)$. Since $\{u_n\}$ is bounded in $\WLGzspace$, there exists $C_3>0$ such that
\begin{equation*}
a(|u_n(t)|) \leq \alpha(\LGnorm[\infty]{u_n})\leq C_3
\end{equation*}
and there exists $C_4>0$ such that
\[
\int_I G(\dot{u}_n)\,dt\leq C_4.
\]
It follows from \eqref{F:growth:Fx} and the above that $\LGnorm[1]{F_x(\cdot,u_n,\dot{u}_n)}$ is uniformly bounded.
Since $u_n\to u$ in $\Lpspace[\infty]$, we get
\[
\left| \int_I \inner{F_x(t,u_n,\dot{u}_n)}{u_n-u}\,dt\right |
\leq
\LGnorm[1]{F_x(\cdot,u_n,\dot{u}_n)} \LGnorm[\infty]{u_n-u}\to 0.
\]
and consequently
\begin{equation}
\label{eq:FVto0}
\lim_{n\to\infty} \int_I \inner{F_v(t,u_n,\dot{u}_n)}{\dot{u}_n-\dot{u}}\,dt = 0.
\end{equation}

By continuity of $F$ we have that $F(t,u_n(t),\pm \dot{u}(t)) \to F(t,u(t),\pm \dot{u}(t))$ a.e. From \eqref{F:growth:F} and $\dot{u}\in \LGspace$ we get
\begin{equation*}
|F(t,u_n(t),\pm\dot{u}(t)|\leq C_3(b(t)+G(\pm \dot{u}(t)) \in \Lpspace[1].
\end{equation*}
Hence
\begin{equation}
\label{eq:F_t_un_dotu_converges}
\lim_{n\to\infty} \int_I F(t,u_n,\pm\dot{u})\,dt = \int_I F(t,u,\pm\dot{u})\,dt.
\end{equation}
From the other hand, convexity of $F(t,x,\cdot)$, \eqref{eq:FVto0} and \eqref{eq:F_t_un_dotu_converges} yields
\begin{equation*}
\limsup_{n\to \infty}\int_I F(t,u_n,\dot{u}_n)\,dt
\leq
\lim_{n\to \infty} \int_I F(t,u_n,\dot{u})
+
\inner{F_v(t,u_n,\dot{u}_n)}{\dot{u}_n-\dot{u}}\,dt
= \int_I F(t,u,\dot{u})\,dt.
\end{equation*}
Since $F(t,u_n(t),\dot{u}_n(t))\geq 0$ and $F(t,u_n(t),\dot{u}_n(t)) \to F(t,u(t),\dot{u}(t)) \text{ a.e.}$, we have
\begin{equation*}
\int_I F(t,u,\dot{u})\,dt \leq \liminf_{n\to \infty}\int_I F(t,u_n,\dot{u}_n)\,dt
\end{equation*}
by Fatou's Theorem.
Finally,
\begin{equation}
\label{eq:F_t_un_dotun_converges}
\lim_{n\to \infty}\int_I F(t,u_n,\dot{u}_n)\,dt = \int_I F(t,u,\dot{u})\,dt.
\end{equation}

Now we are in position to show that $\dot{u}_n\to \dot{u}$ in $\LGspace$. The following is a modification of \cite[Lemma 3.16]{ChmMak17}. Convexity of $F(t,x,\cdot)$ yields
\begin{equation*}
\frac{F(t,u_n(t),\dot{u}_n(t)) + F(t,u_n(t),-\dot{u}(t))}{2} - F\left(t,u_n(t),\frac{\dot{u}_n(t)-\dot{u}(t)}{2}\right)\geq 0.
\end{equation*}
By continuity of $F$, $\dot{u}_n\to \dot{u}$ a.e. and \ref{asm:F:zero} we obtain
\begin{multline*}
\lim_{n\to\infty}\frac{F(t,u_n(t),\dot{u}_n(t)) + F(t,u_n(t),-\dot{u}(t))}{2} - F\left(t,u_n(t),\frac{\dot{u}_n(t)-\dot{u}(t)}{2}\right)
=\\=
\frac{F(t,u(t),\dot{u}(t)) + F(t,u(t),-\dot{u}(t))}{2} \text{ a.e.}
\end{multline*}
Thus, by Fatou's Lemma,
\begin{equation*}
\int_I \frac{F(t,u,\dot{u}) + F(t,u,-\dot{u})}{2} \,dt
\leq
\liminf_{n\to \infty}
\int_I
\frac{F(t,u_n,\dot{u}_n) + F(t,u_n,-\dot{u})}{2} - F\left(t,u_n,\frac{\dot{u}_n-\dot{u}}{2}\right)
\,dt
\end{equation*}
Taking into account \eqref{eq:F_t_un_dotu_converges} and \eqref{eq:F_t_un_dotun_converges} we have
\begin{equation*}
\lim_{n\to \infty}
\int_I
\frac{F(t,u_n,\dot{u}_n) + F(t,u_n,-\dot{u})}{2}
=
\int_I \frac{F(t,u,\dot{u}) + F(t,u,-\dot{u})}{2} \,dt
\end{equation*}
and consequently
\begin{equation*}
\int_I \frac{F(t,u,\dot{u}) + F(t,u,-\dot{u})}{2} \,dt
\leq
\int_I \frac{F(t,u,\dot{u}) + F(t,u,-\dot{u})}{2} \,dt
- \limsup \int_I F\left(t,u_n,\frac{\dot{u}_n-\dot{u}}{2}\right)\,dt
\end{equation*}
It follows that
\begin{equation*}
\lim_{n\to\infty} \int_I F\left(t,u_n,\frac{\dot{u}_n-\dot{u}}{2}\right)\,dt  = 0
\end{equation*}

From ellipticity condition \ref{asm:F:ellipticity} we get
\begin{equation*}
\lim_{n\to\infty} \int_I G\left(\frac{\dot{u}_n-\dot{u}}{2}\right)\,dt
\leq
\lim_{n\to\infty} \frac{1}{\Lambda} \int_I F\left(t,u_n,\frac{\dot{u}_n-\dot{u}}{2}\right)\,dt = 0
\end{equation*}
Thus $\dot{u}_n \to \dot{u}$ in $\LGspace$ by \cite[Theorem 3.13]{ChmMak17}.

\end{proof}

\subsection{Mountain Pass geometry}

Now we show that $\Jcal$ has a mountain pass geometry. It follows  immediately from \eqref{eq:J}, \ref{asm:F:zero} and \ref{asm:V:zero} that

\begin{lemma}
$\Jcal(0)=0$
\end{lemma}

We next prove that $\Jcal$ is negative at some point outside $B_\rho(0)$, where $
\rho = \frac{\rho_0}{C_{\infty,G}}.
$

\begin{lemma}
    \label{lem:J(e)<0}
There exists $e\in \WLGzspace$ such that $\WLGnorm{e}>\rho$ and $\Jcal(e)<0$.
\end{lemma}

\begin{proof}
Choose $u_0\in \WLGzspace$ such that
$
|\{ t\in I\colon |u_0(t)| \geq r_0 \}|>0.
$
Set $M=\sup\{ |V(t,x)| \colon t\in I, |x|\leq r_0\}$.
For any $\lambda > 1$ we have
\begin{multline*}
\Jcal(\lambda u_0) =
\int_I F(t,\lambda u_0,\lambda \dot{u}_0 )\,dt + \int_I V(t,\lambda u_0)\,dt +
\int_I \inner{f}{\lambda u_0}\,dt
\leq \\ \leq
\lambda^{\theta_F }\int_I F(t, u_0, \dot{u}_0 )\,dt
+ \lambda^{\theta_V} \int_{ \{|u_0(t)|\geq r_0\} } V(t, u_0)\,dt + M |I|
+ \lambda \int_I \inner{f}{u_0}\,dt
\end{multline*}
by Lemma \ref{lem:subhomogeneous}. Since $V(t,x)$ is negative for  $|x|\geq r_0$ and $\theta_V>1$, $\theta_F < \theta_V$,
\begin{equation*}
\lim_{\lambda\to \infty} \Jcal(\lambda u_0) = -\infty
\end{equation*}
Thus, choosing $\lambda_0$ large enough, we can set $e=\lambda_0u_0$.
\end{proof}

\begin{lemma}
    \label{lem:J>0}
Assume that either \ref{thmA:asm_ineq} or \ref{thmB:asm_ineq} holds. Then
\begin{equation*}
\inf_{\WLGznorm{u}=\rho} \Jcal(u)>0
\end{equation*}
\end{lemma}
\begin{proof}
Let $\WLGznorm{u} = \rho$. Then
\begin{equation*}
|u(t)|\leq C_{\infty,G} \WLGznorm{u} =\rho_0,\, \text{ for all $t\in I$}.
\end{equation*}
Using \ref{asm:F:ellipticity}, \ref{asm:V:near_zero} and H\"older's inequality we have
\begin{equation*}
\Jcal(u) \geq \Lambda \int_I G(\dot{u})\,dt - \int_I g(t)\,dt -2 \LGastnorm{f} \LGnorm{u}.
\end{equation*}

Assume that \eqref{thmA:asm_ineq} holds. Since $\rho\geq 1$, using Proposition \ref{prop:modular_greater_than_norm} and \eqref{ineq:Poincare}, we have
\begin{equation*}
\Jcal(u) \geq \Lambda \rho - \int_I g(t)\,dt - 2|I|\,\LGastnorm{f} \rho >0
\end{equation*}
by assumption \eqref{thmA:asm_ineq}.

Assume that \eqref{thmB:asm_ineq} holds. If $\rho>1$ then by Proposition \ref{prop:modular_greater_than_norm_globally}
\begin{equation*}
\Jcal(u) \geq \Lambda \rho^{p_G} - \int_I g(t)\,dt - 2|I|\,\LGastnorm{f} \rho
\end{equation*}
Similarly, if $\rho\leq 1$ then
\begin{equation*}
\Jcal(u) \geq \Lambda \rho^{q_G} - \int_I g(t)\,dt - 2|I|\,\LGastnorm{f} \rho
\end{equation*}
From \eqref{thmB:asm_ineq} it follows that in both cases $\Jcal(u)>0$.
\end{proof}


\bibliographystyle{elsarticle-num}
\bibliography{MPT_AniOp}

\end{document}